\documentclass[12pt]{article}
\usepackage{amssymb,amsmath}
\newtheorem{theorem}{Theorem}[section]
\newtheorem{lemma}{Lemma}[section]
\newtheorem{proposition}{Proposition}[section]

\newcommand{\ds}{\displaystyle}
\newcommand{\Sp}{\mathop{\rm Lin\,}\nolimits}

\newenvironment{proof}{{\bf Proof.}}{\par\hspace{25em}\rule{1ex}{1ex}\par}
\title{On geodesics of Berger tangent sphere bundle of Hermitian locally symmetric manifold.}
\author{Alexander Yampolsky}
\date{}
\begin{document}
\maketitle
\begin{abstract}
    We propose a special deformation of the Sasaki metric on tangent and unit tangent bundle
    of a Hermitian locally symmetric manifold. Geodesics of this deformed metric have different
    projections on a base manifold for tangent or unit tangent bundle cases
    in contrast to usual Sasaki metric. Nevertheless,
    the projections of geodesics of the unit tangent bundle
    still preserve the property to have all  geodesic curvatures constant.
    \\[1ex]
    {\it AMS subject class:} Primary 53B20, 53B25; Secondary 53B21.\\[3ex]
    \end{abstract}

\section*{Introduction}

    Let $(M,g)$ be Riemannian manifold. Denote by $TM$ and $T_1M$ tangent and unit tangent
    bundle of $(M,g)$ with Sasaki metric. It is easy to prove that if $\pi$ is
    a bundle projection $\pi:TM\to M$ and $\Gamma(\sigma)$
    is geodesic on $TM$ or $T_1M$ then the projected curve
    $\gamma(\sigma)=\pi\circ\Gamma(\sigma)$ on $M$ is the same. In other words,
    geodesic lines on $TM$ or $T_1M$ are generated by different vector fields along
    {\it the same} set of curves in a base manifold. A complete description of base
    curves and vector fields generating geodesics in the case of base manifold of
    constant curvature one can find in \cite{Sasaki} and \cite{Sato}. It was proved
    that the projected curves have constant (possibly zero) first and second geodesic
    curvatures while the others vanish. P.Nagy \cite{Nagy} generalized
    these results for the case of locally symmetric base  and characterized  the projected
    curves by all constant geodesic curvatures.

    It present paper we propose a special deformation of Sasaki metric for the case of
    Hermitian locally symmetric base manifold which distinguish the projections of
    geodesics in $TM$ and $T_1M$ cases but preserves the property to have constant
    geodesic curvatures in $T_1M$ case.

The general idea is M.Berger-type. Let $S^{2n-1}$ be a unit sphere in Euclidean space $E^{2n}$. Let $J$ be a standard
complex structure on $E^{2n}$. If $N$ is a unit normal vector field on $S^{2n-1}$ then $JN$ is a so-called {\it Hopf}
vector field on $S^{2n+1}$. M.Berger deformation of standard sphere metric assumes its deformation along
    integral trajectories of the Hopf vector field. Consider a Hermitian manifold
    $(M^{2n},g,J)$ and its tangent sphere bundle. Then at each point of $M^{2n}$ the tangent
    sphere $S^{2n-1}$ carries the Hopf vector field. Applying M.Berger metric deformation
    to each tangent sphere we get the tangent sphere bundle over $M^{2n}$ with
    M.Berger metric spheres the fibers. Call it {\it Berger tangent (sphere) bundle}.
    The main result of the paper is the following.

    {\bf Theorem \ref{main}}\ {\it
    Let $\gamma=\pi\circ\Gamma$ be a projection of a curve $\Gamma$ on the Berger tangent
    sphere bundle over Hermitian locally symmetric manifold $M$. Then all geodesic
    curvatu\-res of $\gamma$ are constant.
    }

    If $\Gamma$ is a geodesic on  Berger {\it tangent} bundle  $TM$, then the projected curve
    $\gamma=\pi\circ\Gamma$ does not posess this property.

    In a specific case of of $CP^n$ the Theorem \ref{main} can be improved.

    {\bf Theorem \ref{CP}}\ {\it
    Let $\Gamma$ be a geodesic of the Berger tangent sphere bundle over the complex
    projective space $CP^n$. Then the geodesic curvatures of $\gamma=\pi\circ\Gamma$
    are all constant and $k_6=\dots =k_{n-1}=0$.
    }

{\bf Acknowledgement}: The author thanks P. Nagy (Debrecen, Hungary) for the idea to consider the deformation proposed.
\section{Some general considerations}
    Let $(M,g)$ be $n$-dimensional Riemannian manifold  with metric $g$. Denote by
    $\big<\cdot,\cdot\big>$ a scalar product with respect to $g$.
    A natural Riemannian metric on the tangent bundle has been defined by S. Sasaki.
    We describe it briefly in terms of the {\it connection map}.

    At each point $Q=(q,\xi)\in TM$ the tangent space $T_QTM$ can be split
    into the so-called {\it vertical} and {\it horizontal} parts:
    $$
    T_{Q}TM=\mathcal{H}_{Q}TM \oplus \mathcal{V}_{Q}TM.
    $$
The vertical part  $\mathcal{V}_{Q}TM$ is tangent to the fiber, while the horizontal part is transversal to it.  Denote
$(x^1,\dots,x^n;\xi^1,\dots,\xi^n)$ the natural induced local coordinate system on $TM$. Denote
$\partial_i=\partial/\partial{x^i},\ \partial_{n+i}=\partial/\partial{\xi^i}$. Then for $\tilde X \in T_{Q}TM^n$ we
    have
$$
   \tilde X=\tilde X^i \partial_i + \tilde X^{n+i} \partial _{n+i}.
$$
Denote by  $\pi:TM \to M$ the tangent bundle projection map. Then its differential $\pi_*:T_{ Q}TM \to T_qM $  acts on
$\tilde X$ as $$\ds \pi_*\tilde X=\tilde X^i \partial _i$$  and defines a linear isomorphism between $\mathcal{V}_{
Q}TM$ and $T_qM$.

The so-called {\it connection map} $K: T_{ Q}TM \to T_qM$ acts on $\tilde X$ by the rule $$\ds K\tilde X=(\tilde
X^{n+i}+\Gamma_{jk}^{i}\xi^j\tilde X^k) \partial_i $$  and defines a linear isomorphism between $\mathcal{H}_{Q}TM$ and
$T_qM$. The  images $\pi_*\tilde X$ and  $K\tilde X$ are called {\it horizontal} and  {\it vertical } projections of
$\tilde X$, respectively.  It is easy to see that  $\mathcal{V}_{ Q}=\ker\pi_*|_{ Q},\ \mathcal{H}_{Q}=\ker K|_{Q}$.

Let $\tilde X,\tilde Y \in T_{ Q}TM.$
    The standard {\it Sasaki metric} on $TM$ is defined by the following scalar product
    $$
        \big<\big< \tilde X,\tilde Y \big>\big>\big|_Q=
        \big<\pi_* \tilde X, \pi_* \tilde Y\big>\big|_q+\big<K \tilde X,K \tilde Y\big>\big|_q
    $$
    at each point $Q=(q,\xi)$.
    Horizontal and vertical subspaces are mutually orthogonal with respect to Sasaki
    metric.

    The operations inverse to projections are called {\it lifts}. Namely, if
    $X \in T_qM^n$, then
    $$\ds X^h=X^i \partial_i -\Gamma_{jk}^i\xi^j X^k \partial_{n+i}$$
    is in $\mathcal{H}_{Q}TM$ and is called the {\it horizontal lift } of  X, and
    $$ \ds X^v=X^i \partial_{n+i} $$
    is in $\mathcal{V}_{Q}TM$ and is called the {\it vertical lift } of  $ X$.

    The Sasaki metric can be completely defined by scalar product of various
    combinations of lifts of vector fields from $M$ to $TM$ as
    $$
   \big<\big<X^h,Y^h\big>\big>=\big<X,Y\big>, \ \
   \big<\big<X^h,Y^v\big>\big>=0,             \ \
    \big<\big<X^v,Y^v\big>\big>=\big<X,Y\big>.
   $$

   Consider now as a base manifold $M$ a Hermitian manifold
   $(M^{2n},g,J)$. Define a  {\it deformation} of Sasaki metric along the $J\xi$
   directions in each fiber of the form
\begin{equation}\label{metric}
\begin{array}{l}
    \big<\big<X^h,Y^h\big>\big>=\big<X,Y\big>, \\
    \big<\big<X^h,Y^v\big>\big>=0,              \\
    \big<\big<X^v,Y^v\big>\big>=\big<X,Y\big>+\delta^2\big<X,J\xi\big>\big<Y,J\xi\big>,
\end{array}
\end{equation}
    where $J$ is an almost complex structure on $M$ and $\delta$ is some constant.

    Geometrically, this deformation means that we deform each tangent sphere
    $S^{2n-1}\subset T_qM$ along the fibers of standard Hopf fibration of $S^{2n-1}$
    at each point $q\in M$. We will refer to the tangent (sphere) bundle with
    the metric (\ref{metric}) as {\it Berger tangent (sphere) bundle}.

    In what follows we suppose that $M$ is { Hermitian locally symmetric} space. In
    this case $M$ is necessarily Kahlerian, i.e. $\nabla J=0$, and locally symmetric as a
    Riemannian manifold \cite[Proposition 9.1]{KN}.

    The following formulas are independent on the choice of tangent bundle metric and
    are known as Dombrowski formulas.
\begin{lemma} At each point $(q,\xi)\in TM$  the brackets of lifts of vector fields
    from $M$ to $TM$ are
    $$
\big[X^h,Y^h\big]=\big[X,Y\big]^h-\big(R(X,Y)\xi\big)^v, \quad \big[X^h,Y^v\big]=\big(\nabla_XY\big)^v, \quad
    \big[X^v,Y^v\big]=0,
    $$
    where $\nabla$ is the connection on $M$ and $R$ its curvature tensor.
\end{lemma}

    Denote by $\tilde \nabla$ the Levi-Civita connection of the metric (\ref{metric}).
    The following Kowalski-type lemma is the main tool for further considerations.

\begin{lemma}\label{Kow}

    The Levi-Civita connection of metric(\ref{metric}) is
    completely  defined at each point $(q,\xi)\in TM$ by
    $$
    \begin{array}{l}
    \tilde\nabla_{X^h}Y^h=\big(\nabla_XY\big)^h-\frac12 \big(R(X,Y)\xi\big)^v, \\[1ex]
    \tilde\nabla_{X^h}Y^v=\frac12\Big(R(\xi,Y)X+\delta^2\left<Y,J\xi\right>R(\xi,J\xi)X\Big)^h+
    \Big(\nabla_XY\Big)^v, \\ [1ex]
    \tilde\nabla_{X^v}Y^h=\frac12\Big(R(\xi,X)Y+\delta^2\left<X,J\xi\right>R(\xi,J\xi)Y\Big)^h
    \\[1ex]
    \tilde\nabla_{X^v}Y^v=\delta^2\Big(\left<X,J\xi\right>JY+\left<Y,J\xi\right>JX -\\[1ex]
    \qquad\qquad\frac{\delta^2}{1+\delta^2|\xi|^2}\big(\left<Y,\xi\right>\left<X,J\xi\right>+
    \left<X,\xi\right>\left<Y,J\xi\right>\big)J\xi\Big)^v,
    \end{array}
    $$
    where $\nabla$ is the Levi-Civita connection on $M$ and $R$ is its curvature
    tensor.
\end{lemma}

    \begin{proof}
    To prove this lemma we will need the useful formulas which we naturally
    gather in a separate sublemma.
\begin{lemma}\label{deriv}
    The following rules of differentiations are true:
    $$
    \begin{array}{l}
        X^h\big<\big<Y^h,Z^h\big>\big>=\big<\nabla_XY,Z\big>+
        \big<Y,\nabla_XZ\big> ,\\
        X^h\big<\big<Y^v,Z^v\big>\big>=\big<\big<(\nabla_XY)^v,Z^v\big>\big>+
        \big<\big<Y^v,(\nabla_XZ)^v\big>\big> ,\\
        X^v\big<\big<Y^h,Z^h\big>\big>=0,\\
        X^v\big<\big<Y^v,Z^v\big>\big>=\delta^2\Big(\big<Y,JX\big>\big<Z,J\xi\big>+
        \big<Y,J\xi\big>\big<Z,JX\big>\Big),
    \end{array}
    $$
    where $\left<\cdot,\cdot\right>$ means the scalar product with respect to metric
    of the base manifold.
\end{lemma}

    \begin{proof}
    \begin{itemize}
    \item[i)]
    Indeed, keeping in mind (\ref{metric}), we have
    $$
    X^h\big<\big<Y^h,Z^h\big>\big>=X^h\big<Y,Z\big>=\big<\nabla_XY\big>+\big<Y,\nabla_XZ\big>.
    $$
    \item[ii)] In a similar way
    $$
    \begin{array}{l}
    X^h\big<\big<Y^v,Z^v\big>\big>=X^h\Big(
    \big<Y,Z\big>+\delta^2\big<Y,J\xi\big>\big<Z,J\xi\big>\Big)=\\
    \big<\nabla_XY,Z\big>+\big<Y,\nabla_XZ\big>+\delta^2X^h\Big(\big<Y,J\xi\big>\big<Z,J\xi\big>\Big).
    \end{array}
    $$
    Since $M$ is Kahlerian, $\nabla_XJ=0$ and we have
    $$
    \begin{array}{l}
    X^h\big<Y,J\xi\big>=-X^h\big<JY,\xi\big>=
    -X^i\partial_i\big<JY,\xi\big>+\Gamma^s_{jk}\xi^jX^k\partial_{n+s}\big<JY,\xi\big>=\\[2ex]
    -X^i\big<J\nabla_iY,\xi\big>-X^i\xi^k\big<JY,\Gamma^s_{ki}\partial_s\big>+
    \Gamma^s_{ki}\xi^kX^i\big<JY,\partial_s\big>= \big<\nabla_XY,J\xi\big>.
    \end{array}
    $$
    Therefore,
    $$
    \begin{array}{l}
    X^h\big<\big<Y^v,Z^v\big>\big>=\big<\nabla_XY,Z\big>+\big<Y,\nabla_XZ\big>+
        \delta^2\big<\nabla_XY,J\xi\big>\big<Z,J\xi\big>+\\[2ex]
        \delta^2\big<Y,J\xi\big>\big<\nabla_XZ,J\xi\big>=
         \big<\big<(\nabla_XY)^v,Z^v\big>\big>+\big<\big<Y^v,(\nabla_XZ)^v\big>\big>.
    \end{array}
    $$
    \item[iii)] Rather evident, that
    $X^v\big<\big<Y^h,Z^h\big>\big>=X^v\big<Y,Z\big>=0$.
    \item[iv)] Finally, it is easy to see that
    $X^v\big<Y,J\xi\big>=X^i\partial_{n+i}\big<Y,J\xi\big>=\big<Y,JX\big>$ and
    therefore
    $$
    \begin{array}{l}
    X^v\big<\big<Y^v,Z^v\big>\big>=
    X^v\Big(\big<Y,Z\big>+\delta^2\big<Y,J\xi\big>\big<Z,J\xi\big>\Big)=\\
    \delta^2\Big(\big<Y,JX\big>\big<Z,J\xi\big>+\big<Y,J\xi\big>\big<Z,JX\big>\Big).
    \end{array}
    $$
    \end{itemize}
    \end{proof}

    Now we can prove the lemma relatively easy applying the Kozsul formula for the
    Levi-Civita connection
    $$
    \begin{array}{rl}
    2\big<\nabla_AB,C\big>=&A\big<B,C\big>+B\big<A,C\big>-C\big<A,B\big>+\\
    &\big<[A,B],C\big>+\big<[C,A],B\big>-\big<[B,C],A\big>
    \end{array}
    $$
    and Dombrowski formulas to the metric (\ref{metric}).

    (i) Setting $A=X^h, B=Y^h, C=Z^h$ we see that
    $$
        2\big<\big<\tilde\nabla_{X^h}Y^h,Z^h\big>\big>=2\big<\nabla_XY,Z\big>=
            2\big<\big<\,(\nabla_XY)^h,Z^h\big>\big>.
    $$
    Setting $A=X^h, B=Y^h, C=Z^v$, we have
    $$
    \begin{array}{rr}
    2\big<\big<\tilde\nabla_{X^h}Y^h,Z^v\big>\big>=&-Z^v\big<\big<X^h,Y^h\big>\big>+
    \big<\big<[X^h,Y^h],Z^v\big>\big>=\\[1ex]
    &-\big<\big<(R(X,Y)\xi)^v,Z^v\big>\big>.
    \end{array}
    $$
    Hence
    $$
    \tilde\nabla_{X^h}Y^h=\big(\nabla_XY\big)^h-\frac12 \big(R(X,Y)\xi\big)^v
    $$

    (ii) Set $A=X^h, B=Y^v,C=Z^h$. Then
    $$
    \begin{array}{rl}
    2\big<\big<\tilde\nabla_{X^h}Y^v,Z^h\big>\big>=&\big<\big<[Z^h,X^h],Y^v\big>\big>=
    \big<\big<(R(X,Z)\xi)^v,Y^v\big>\big>= \\[1ex]
    &\big<R(X,Z)\xi,Y\big>+ \delta^2\big<R(X,Z)\xi,J\xi\big>\big<Y,J\xi\big>=\\[1ex]
    &\big<R(\xi,Y)X,Z\big>+\delta^2\big<Y,J\xi\big>\big<R(\xi,J\xi)X,Z\big>=\\[1ex]
    &\big<\big<\big(R(\xi,Y)X+\delta^2\big<Y,J\xi\big>R(\xi,J\xi)X\big)^h,Z^h\big>\big>
    \end{array}
    $$

    Set $A=X^h, B=Y^v,C=Z^v$. Then, applying lemma \ref{deriv}, we have
    $$
    \begin{array}{l}
    2\big<\big<\tilde\nabla_{X^h}Y^v,Z^v\big>\big>=\\[1ex]
    \hspace{7em} X^h\big<\big<Y^v,Z^v\big>\big>+
    \big<\big<[X^h,Y^v],Z^v\big>\big>+\big<\big<[Z^v,X^h],Y^v\big>\big>=\\[1ex]
     \hspace{7em} \big<\big<(\nabla_XY)^v,Z^v\big>\big>+\big<\big<Y^v,(\nabla_XZ)^v\big>\big>+
    \big<\big<(\nabla_XY)^v,Z^v\big>\big>- \\[1ex]
     \hspace{7em} \big<\big<(\nabla_XZ)^v,Y^v\big>\big>=2\big<\big<(\nabla_XY)^v,Z^v\big>\big>.
    \end{array}
    $$
    So we see that
    $$
    \tilde\nabla_{X^h}Y^v=\frac12\Big(R(\xi,Y)X+\delta^2\left<Y,J\xi\right>R(\xi,J\xi)X\Big)^h+
    \Big(\nabla_XY\Big)^v.
    $$

    (iii) Set $A=X^v, B=Y^h,C=Z^h$. Then
    $$
    \begin{array}{l}
    2\big<\big<\tilde\nabla_{X^v}Y^h,Z^h\big>\big>=
    X^v\big<\big<Y^h,Z^h\big>\big>+\big<\big<[X^v,Y^h],Z^h\big>\big>+\\[1ex]
    \qquad \big<\big<[Z^h,X^v],Y^h\big>\big>-\big<\big<[Y^h,Y^h],X^v\big>\big>=
    \big<\big<(R(Y,Z)\xi)^v,X^v\big>\big>=\\[1ex]
    \qquad\big<R(Y,Z)\xi,X\big>+ \delta^2\big<R(Y,Z)\xi,J\xi\big>\big<X,J\xi\big>=\\[1ex]
    \qquad\big<R(\xi,X)Y,Z\big>+\delta^2\big<X,J\xi\big>\big<R(\xi,J\xi)Y,Z\big>=\\[1ex]
    \qquad\big<\big<(R(\xi,X)Y+\delta^2\big<X,J\xi\big>R(\xi,J\xi)Y)^h,Z^h\big>\big>
    \end{array}
    $$

        Set $A=X^v, B=Y^h,C=Z^v$. Then
    $$
    \begin{array}{l}
    2\big<\big<\tilde\nabla_{X^v}Y^h,Z^v\big>\big>=\\[1ex]
    Y^h \big<\big<Z^v,X^v\big>\big>+
    \big<\big<[X^v,Y^h],Z^v\big>\big>-\big<\big<[Y^h,Z^v],X^v\big>\big>=
    \big<\big<(\nabla_YZ)^v,X^v\big>\big>+\\[1ex]
    \big<\big<Z^v,(\nabla_YX)^v\big>\big>-
    \big<\big<(\nabla_YX)^v,Z^v\big>\big>-\big<\big<(\nabla_YZ)^v,X^v\big>\big>=0
    \end{array}
    $$

    So, we have
    $$
    \tilde\nabla_{X^v}Y^h=\frac12\Big(R(\xi,X)Y+\delta^2\left<X,J\xi\right>R(\xi,J\xi)Y\Big)^h
    $$

    (iv) Setting $A=X^v, B=Y^v,C=Z^h$, we have
    $$
    \begin{array}{l}
        2\big<\big<\tilde\nabla_{X^v}Y^v,Z^h\big>\big>=-Z^h\big<\big<X^v,Y^v\big>\big>+
        \big<\big<[Z^h,X^v],Y^v\big>\big>-\big<\big<[Y^v,Z^h],X^v\big>\big>= \\[1ex]
        -\big<\big<(\nabla_ZX)^v,Y^v\big>\big>-\big<\big<X^v,(\nabla_XY)^v\big>\big>+
        \big<\big<(\nabla_ZX)^v,Y^v\big>\big>+\big<\big<(\nabla_ZY)^v,X^v\big>\big>=0
    \end{array}
    $$

    Set, finally, $A=X^v, B=Y^v,C=Z^v$. Then
    $$
    \begin{array}{l}
        2\big<\big<\tilde\nabla_{X^v}Y^v,Z^v\big>\big>=X^v\big<\big<Y^v,Z^v\big>\big>+
        Y^v\big<\big<X^v,Z^v\big>\big>-Z^v\big<\big<X^v,Y^v\big>\big>= \\[1ex]
        \hspace{2cm}\delta^2\Big( \big<Y,JX\big>\big<Z,J\xi\big> + \big<Y,J\xi\big>\big<Z,JX\big>+
        \big<X,JY\big>\big<Z,J\xi\big>\\[1ex]
        \hspace{2cm}+\big<X,J\xi\big>\big<Z,JY\big>-\big<X,JZ\big>\big<Y,J\xi\big>-
        \big<X,J\xi\big>\big<Y,JZ\big>\Big)=\\[1ex]
        \hspace{2cm}2\delta^2\Big(\big<Y,J\xi\big>\big<JX,Z\big>+\big<X,J\xi\big>\big<JY,Z\big>\Big).
    \end{array}
    $$

    Thus, we see that
    $$
    \big<\big<\tilde\nabla_{X^v}Y^v,Z^v\big>\big>=
    \delta^2\Big(\big<Y,J\xi\big>\big<JX,Z\big>+\big<X,J\xi\big>\big<JY,Z\big>\Big).
    $$
    On the other hand,
    $$
    \big<\big<(JY)^v,Z^v\big>\big>=\big<JY,Z\big>+\delta^2\big<Y,\xi\big>\big<Z,J\xi\big>
    $$
    and
    $$
    \big<\big<(J\xi)^v,Z^v\big>\big>=\big<J\xi,Z\big>+\delta^2\big<Z,J\xi\big>|\xi|^2=
    (1+\delta^2|\xi|^2)\big<Z,J\xi\big>.
    $$
    Therefore,
    $$
    \big<Z,J\xi\big>=\frac{1}{1+\delta^2|\xi|^2}\big<\big<(J\xi)^v,Z^v\big>\big>
    $$
    and as a consequence
    $$
    \begin{array}{rl}
    \big<JY,Z\big>=& \big<\big<(JY)^v,Z^v\big>\big>-
    \delta^2\big<Y,\xi\big>\frac{1}{1+\delta^2|\xi|^2}\big<\big<(J\xi)^v,Z^v\big>\big>=\\[1ex]
    &\big<\big<(JY)^v-\frac{\delta^2}{1+\delta^2|\xi|^2}\big<Y,\xi\big>(J\xi)^v,Z^v\big>\big>.
    \end{array}
    $$
    So we have
    $$
    \begin{array}{rl}
    \big<\big<\tilde\nabla_{X^v}Y^v,Z^v\big>\big>=&\delta^2
    \big<\big<\left[\big<X,J\xi\big>\Big( JY-\frac{\delta^2}{1+\delta^2|\xi|^2}\big<Y,\xi\big>J\xi\Big)\right.+\\[1ex]
    &\left.\big<Y,J\xi\big>\Big(JX-\frac{\delta^2}{1+\delta^2|\xi|^2}\big<X,\xi\big>J\xi\Big)\right]^v,Z^v\big>\big>.
    \end{array}
    $$
    Finally, we conclude, that
    $$
    \begin{array}{rl}
    \tilde\nabla_{X^v}Y^v=&\delta^2\Big(\big<X,J\xi\big>JY+\big<Y,J\xi\big>JX-\\[1ex]
    &\frac{\delta^2}{1+\delta^2|\xi|^2}\big(\big<Y,\xi\big>\big<X,J\xi\big>+
    \big<X,\xi\big>\big<Y,J\xi\big>\big)J\xi\Big)^v.
    \end{array}
    $$
  \end{proof}

\section{Geodesics of the deformed metric. }

    Consider a curve $\Gamma$ on the tangent bundle with the metric (\ref{metric}).
    Geometrically, $\Gamma=\{x(\sigma),\xi(\sigma)\}$, where $x(\sigma)$ is a curve on
    $M$ and $\xi(\sigma)$ is a vector field along this curve. Let $\sigma$ be an arc
    length parameter on $\Gamma$. Then $\Gamma'=\left(\frac{dx}{d\sigma}\right)^h+
    (\nabla_{\frac{dx}{d\sigma}}\xi)^v.$
    Introduce the notations $x'=\frac{dx}{d\sigma}$ and
    $\xi'=\nabla_{\frac{dx}{d\sigma}}\xi$. Then
    $$
    \Gamma'=(x')^h+(\xi')^v.
    $$

    Using the Lemma \ref{Kow} we can easily derive the differential equations of geodesic
    lines of the metric (\ref{metric}).
    \begin{lemma}
    Let $(M^{2n},g,J)$ be Hermitian locally symmetric manifold and $TM$ its Berger
    tangent bundle.
    A curve $\Gamma=\{x(\sigma),\xi(\sigma)\}$ is a geodesic  $TM$ if $x(\sigma)$ and
    $\xi(\sigma)$ satisfy the equations
    \begin{equation}\label{geo}
    \begin{array}{l}
    x''+\mathcal{R}(\xi,\xi')x'=0\\[1ex]
    \xi''+2\delta^2\big<\xi',J\xi\big>\Big(J\xi'-\frac{\delta^2}{1+\delta^2|\xi|^2}\big<\xi',\xi\big>J\xi\Big)=0,
    \end{array}
    \end{equation}
    where $\mathcal{R}(\xi,\xi')=R(\xi,\xi')+\delta^2\big<\xi',J\xi\big>R(\xi,J\xi)$ and
    $R$ is the curvature operator of the base manifold $M$.
    \end{lemma}

    Consider now the tangent sphere bundle $T_1M$. The unit normal to $T_1M$ is
    $\xi^v$. Indeed, with respect to metric(\ref{metric}) we have
    $$
    \begin{array}{l}
    \big<\big<X^h,\xi^v\big>\big>=0 \mbox{ \quad for all $X$ tamgent to $M$},\\
    \big<\big<X^v,\xi^v\big>\big>=0 \mbox{ \quad for all $X\in \xi^\perp$ }.
    \end{array}
    $$
    So, to obtain the equations of geodesics for $T_1M$, it is sufficient to set
    $|\xi|=1$ in (\ref{geo}) and to suppose the second equation left-hand side of
    (\ref{geo}) to be proportional to $\xi$. Thus, we get
    \begin{lemma}
    Let $(M^{2n},g,J)$ be Hermitian locally symmetric manifold and $T_1M$ its Berger
    tangent sphere bundle. Set $c=|\xi'|, \ \mu=\big<\xi',J\xi\big>$.
    A curve $\Gamma=\{x(\sigma),\xi(\sigma)\}$ is a geodesic on $T_1M$ if and only if
    (a) $c=const, \mu=const$; (b) $x(\sigma)$ and $\xi(\sigma)$ satisfy
    the equations
    \begin{equation}\label{geo1}
    \begin{array}{l}
    x''+\mathcal{R}(\xi,\xi')x'=0\\[1ex]
    \xi''+c^2\xi+2\delta^2\mu (J\xi'+\mu\xi)=0.
    \end{array}
    \end{equation}
    where $\mathcal{R}(\xi,\xi')=R(\xi,\xi')+\delta^2\mu R(\xi,J\xi)$ and
    $R$ is the curvature operator of the base manifold $M$.
    \end{lemma}

    \begin{proof}
    Set $|\xi|=1$ in (\ref{geo}) and suppose that
\begin{equation}\label{st}
    \xi''+2\delta^2\big<\xi',J\xi\big>J\xi'=\rho\xi,
\end{equation}
    where $\rho$ is some function.

    Set $c=|\xi'|$. Then $c=const$, since
    directly from \eqref{st} we see that $\big<\xi'',\xi'\big>=0$. Set
    $\mu=\big<\xi',J\xi\big>$. Then $\mu=const$ since $\mu'=\big<\xi'',J\xi\big>=0$ by
    the similar reason. Multiplying \eqref{st} by $\xi$, we found that
    $-\rho=c^2+2\delta^2\mu^2=const$. After substitution of $\rho$ into \eqref{st} we get
    what was claimed.

    \end{proof}

    The difference in description of solutions of (\ref{geo}) and (\ref{geo1}) becomes
    clear
    because of different behaviour of the operator $\mathcal{R}(\xi,\xi')$ along the
    $\pi\circ\Gamma$.
    \begin{proposition}\label{R}
    Let $\gamma=\pi\circ\Gamma$ be a projection of a curve $\Gamma$ on the Berger tangent
    (sphere) bundle over Hermitian locally symmetric manifold $M$. Then
    $\mathcal{R}(\xi,\xi')$is parallel along $\gamma$ for the case of $T_1M$ and
    non-parallel for the case of $TM$.
    \end{proposition}

    \begin{proof}
    Consider the case of $T_1M$ first. Then using (\ref{geo1}) we get
        $$
        \begin{array}{ll}
        \mathcal{R}'(\xi,\xi')=&R(\xi,\xi'')+\delta^2\mu R(\xi',J\xi)+\delta^2\mu R(\xi,J\xi')= \\[1ex]
        &-2\delta^2\mu R(\xi,J\xi')-\delta^2\mu R(J\xi',\xi)+\delta^2\mu R(\xi,J\xi')=0
        \end{array}
        $$
        Here we also used the fact that $R(JX,JY)=R(X,Y)$.

    A similar but slightly longer calculation shows that for the case of $TM$
    $$
    \mathcal{R}'(\xi,\xi')=
    \frac{2\delta^6\big<\xi',J\xi\big>\big<\xi',\xi\big>\big(1-|\xi|^2\big)}{1+\delta^2|\xi|^2}R(\xi,J\xi)
    $$
    which completes the proof.
    \end{proof}

    \begin{theorem}\label{main}
    Let $\gamma=\pi\circ\Gamma$ be a projection of a curve $\Gamma$ on the Berger tangent
    sphere bundle over Hermitian locally symmetric manifold $M$. Then all geodesic
    curvatures of $\gamma$ are constant.
    \end{theorem}
    \begin{proof}

    For the case of $T_1M$ the proposition \ref{R} imply that if $\Gamma$ is geodesic on $T_1M$
    than along each curve $\gamma=\pi\circ\Gamma$
    \begin{equation}\label{deriv1}
    x^{(p+1)}(\sigma)=-\mathcal{R}(\xi,\xi')\, x^{(p)}(\sigma) \mbox{\quad  $p\geq 1$},
    \end{equation}
    or, continuing the process,
    \begin{equation}\label{deriv2}
    x^{(p+1)}(\sigma)=(-1)^{p} \mathcal{R}^p(\xi,\xi')\, x'(\sigma) \mbox{\quad  $p\geq 1$}.
    \end{equation}
    On the other hand, rather evident that
    $$
    \big<\mathcal{R}(\xi,\xi')X,Y\big>=-\big<\mathcal{R}(\xi,\xi')Y,X\big>.
    $$
    This fact and (\ref{deriv1}) imply
    \begin{equation}\label{norm}
    |x^{(p)}(\sigma)|=const \quad \mbox{for all $p\geq 1$}
    \end{equation}
    Indeed,
    $$
    {\textstyle\frac{d}{d\sigma}}\,|x^{(p)}(\sigma)|{\,^2}=2\,\big<x^{(p+1)}(\sigma),x^{(p)}(\sigma)\big>=
    -2\,\big<\mathcal{R}(\xi,\xi')\,x^{(p)}(\sigma),x^{(p)}(\sigma)\big>=0.
    $$

    Denote $s$ the natural parameter on $\gamma$. Then
    $x'_\sigma=x'_s\frac{ds}{d\sigma}$ and therefore
    $$
    1=\|\Gamma'\|^2=\Big|\frac{ds}{d\sigma}\Big|^2+|\xi'|^2+\delta^2\big<\xi',J\xi\big>^2=
    \Big|\frac{ds}{d\sigma}\Big|^2+c^2+\delta^2\mu^2.
    $$
    From this we get
    \begin{equation}\label{arc}
    \frac{ds}{d\sigma}=\sqrt{1-c^2-\delta^2\mu^2}=\sqrt{1-\lambda^2},
    \end{equation}
    where we set $\lambda^2=c^2+\delta^2\mu^2=const$.

    Denote $\nu_1,\dots,\nu_{2n-1}$ the Frenet frame along $\gamma$ and $k_1,\dots,k_{2n-1}$
    the geodesic curvatures of $\gamma$. Then, keeping in mind (\ref{arc}), we have
    $$
    \begin{array}{l}
    x'=\sqrt{1-\lambda^2}\,\nu_1, \\[1ex]
    x''={(1-\lambda^2)}k_1\nu_2.
   \end{array}
    $$
    Now (\ref{norm}) imply $k_1=const$. So next we have
    $$
    x^{(3)}=(1-\lambda^2)^{3/2}\,k_1(-k_1\nu_1+k_2\nu_3)
    $$
    and (\ref{norm}) imply again $k_2=const$. Continuing the process we finish the
    proof.
    \end{proof}

    As it was proved in \cite{Yamp}, for the
    case of $T_1CP^n$ and $TCP^n$ with Sasaki metric the curvatures of
    $\gamma=\pi\circ\Gamma$ are zeroes starting from $k_6$. It is rather remarkable
    that this property still valid for the case of Berger tangent sphere bundle over
    $CP^n$.
    \begin{theorem}\label{CP}
    Let $\Gamma$ be a geodesic of the Berger tangent sphere bundle over the complex
    projective space $CP^n$. Then the geodesic curvatures of $\gamma=\pi\circ\Gamma$
    are all constant and $k_6=\dots =k_{2n-1}=0$.
    \end{theorem}
    \begin{proof}
    For the case of $CP^n$ we have
    $$
    \begin{array}{ll}\ds
    R(X,Y)Z=\frac{m}{4}\Big(&\big<Y,Z\big>X-\big<X,Z\big>Y+\\[1ex]
    &\big<JY,Z\big>JX-\big<JX,Z\big>JY+2\big<X,JY\big>JZ\ \ \Big).
    \end{array}
    $$
    Therefore, $R(\xi,J\xi)=-2J$ and for the case of Berger tangent sphere bundle
    $$
    \mathcal{R}(\xi,\xi')=R(\xi,\xi')-2\delta^2J.
    $$
    Since $R$ and $J$ commute, i.e. $RJ=JR$, the operators $\mathcal{R}$ and $J$ also
    commute. Using this fact one can relatively easy to find expression for powers of
    the operator $\mathcal{R}$ along $\gamma$. Indeed, in \cite{Yamp} it was proved
    that the powers of the curvarure operartor of $CP^n$ satisfy the relations
    $$
    R^{2q}=\Sp(R^2,JR,E),\qquad
    R^{2q+1}=\Sp(JR^2,R,J),
    $$
    where $Lin$ means a linear combination of corresponding tensors and $E$ means the
    identity opeator. It is elementary to see that
    $$
    \begin{array}{l}
    \Sp(R,J)\cdot\Sp(R,J)=\Sp(R^2,JR,E)\\[1ex]
    \Sp(R^2,JR,E)\cdot \Sp(R,J)=\Sp(JR^2,R,J)\\[1ex]
    \Sp(JR^2,R,J)\cdot\Sp(R,J)=\Sp(R^2,JR,E).
    \end{array}
    $$
    Therefore, for the operator $\mathcal{R}$ we also have
    \begin{equation}\label{powers}
    \mathcal{R}^{2q}=\Sp(R^2,JR,E),\qquad
    \mathcal{R}^{2q+1}=\Sp(JR^2,R,J).
    \end{equation}
    On the other hand, $\mathcal{R}J=J\mathcal{R}$ and $R=\mathcal{R}-2\delta^2J$.
    Therefore,
    $$
    R^2=\Sp(\mathcal{R}^2,J\mathcal{R},E)\quad
    JR=\Sp(J\mathcal{R},E)\quad
    JR^2=\Sp(J\mathcal{R},\mathcal{R},J).
    $$
    Taking this into account, we may rewrite (\ref{powers}) as
    \begin{equation}\label{powers2}
    \mathcal{R}^{2q}=\Sp(\mathcal{R}^2,J\mathcal{R},E),\qquad
    \mathcal{R}^{2q+1}=\Sp(J\mathcal{R}^2,\mathcal{R},J).
    \end{equation}
    Using now (\ref{deriv1}),(\ref{deriv2}) and (\ref{powers2}), we get
    \begin{equation}\label{deriv3}
    \begin{array}{l}
    x^{(2q)}=\Sp(Jx''',x'',Jx')\\[1ex]
    x^{(2q+1)}=\Sp(x''',Jx'',x')
    \end{array}
    \end{equation}
    for $q\geq2$. On the other hand, Frenet formulas yield
    $$
    x'=\sqrt{1-\lambda^2}\,\nu_1\quad
    x''=(1-\lambda^2)k_1\nu_2,\quad
    x'''=(1-\lambda^2)^{3/2}( -k_1^2\nu_1+k_1k_2\nu_3)
    $$
    and in general,
    \begin{equation}\label{Frenet}
    \begin{array}{l}
    x^{(2q)}=\Sp(\nu_2,\dots,\nu_{2q-2})+(1-\lambda^2)^{q}\,k_1\dots k_{2q-1}\,\nu_{2q}\\[1ex]
    x^{(2q+1)}=\Sp(\nu_1,\dots,\nu_{2q-1})+(1-\lambda^2)^{q+1/2}\,k_1\dots k_{2q}\,\nu_{2q+1}
    \end{array}
    \end{equation}
    Thus, (\ref{deriv3}) takes the form
    $$
    x^{(2q)}=\Sp(J\nu_1,J\nu_3,\nu_2),\qquad
    x^{(2q+1)}=\Sp(\nu_1,\nu_3,J\nu_2).
    $$
    Comparing the results, for all $q\geq2$ we have
    $$
    \begin{array}{l}
    \Sp(J\nu_1,J\nu_3,\nu_2)=\Sp(\nu_2,\dots,\nu_{2q-2})+(1-\lambda^2)^{q}\,k_1\dots k_{2q-1}\,\nu_{2q},\\[1ex]
    \Sp(\nu_1,\nu_3,J\nu_2)=\Sp(\nu_1,\dots,\nu_{2q-1})+(1-\lambda^2)^{q+1/2}\,k_1\dots k_{2q}\,\nu_{2q+1}. \\[1ex]
    \end{array}
    $$

    Setting $q=2$ and $q=3$, from the second equation above, we get
    $$\left\{
    \begin{array}{l}
    \Sp(\nu_1,\nu_3,J\nu_2)=(1-\lambda^2)^{5/2}\,k_1\dots k_4\,\nu_5 \\
    \Sp(\nu_1,\nu_3,\nu_5,J\nu_2)=(1-\lambda^2)^{7/2}\,k_1\dots k_6\,\nu_7,
    \end{array}
    \right.
    $$
    and therefore, in general,
    $$
    \Sp(\nu_1,\nu_3,\nu_5)=k_1\dots k_6\,\nu_7.
    $$
    Since $\nu_1,\dots,\nu_7$ are linearly independent, we conclude that $k_6=0$.
    \end{proof}

\end{document}